\documentclass[12pt]{article}
\usepackage{graphicx}
\usepackage{amsmath}
\usepackage{amsfonts}
\usepackage{amsthm}
\usepackage[T1]{fontenc}
\usepackage{url}
\usepackage{color}
\usepackage[czech]{babel}
\usepackage[margin=1in]{geometry}
\makeatletter
\renewcommand*\env@matrix[1][\arraystretch]{%
  \edef\arraystretch{#1}%
  \hskip -\arraycolsep
  \let\@ifnextchar\new@ifnextchar
  \array{*\c@MaxMatrixCols c}}
\makeatother
\usepackage{amssymb,bm}
\usepackage{amsmath}
\usepackage{amsthm}
\usepackage{graphicx}
\usepackage[active]{srcltx} 
\usepackage{hyperref}
\hypersetup{pdfborder=0 0 0}

\newtheorem{theorem}{{\sc Theorem}}[section]

\newtheorem{lemma}[theorem]{{\sc Lemma}}

\newtheorem{remark}[theorem]{Remark}

\newtheorem{definition}[theorem]{Definition}

\newtheorem{conjecture}[theorem]{Conjecture}

\newcommand{\cof}{\mathrm{cof}}
\newcommand{\rank}{\mathrm{rank}}

\def\XXint#1#2#3{{\setbox0=\hbox{$#1{#2#3}{\int}$ }
\vcenter{\hbox{$#2#3$ }}\kern-.6\wd0}}

\bmdefine\BGa{\alpha}
\bmdefine\BGb{\beta}
\bmdefine\BGd{\delta}
\bmdefine\BGe{\epsilon}
\bmdefine\BGve{\varepsilon}
\bmdefine\BGf{\phi}
\bmdefine\BGvf{\varphi}
\bmdefine\BGg{\gamma}
\bmdefine\BGc{\chi}
\bmdefine\BGi{\iota}
\bmdefine\BGk{\kappa}
\bmdefine\BGl{\lambda}
\bmdefine\BGn{\eta}
\bmdefine\BGm{\mu}
\bmdefine\BGv{\nu}
\bmdefine\BGp{\pi}
\bmdefine\BGth{\theta}
\bmdefine\BGvth{\vartheta}
\bmdefine\BGr{\rho}
\bmdefine\BGvr{\varrho}
\bmdefine\BGs{\sigma}
\bmdefine\BGvs{\varsigma}
\bmdefine\BGt{\tau}
\bmdefine\BGj{\tau}
\bmdefine\BGu{\upsilon}
\bmdefine\BGo{\omega}
\bmdefine\BGx{\xi}
\bmdefine\BGy{\psi}
\bmdefine\BGz{\zeta}
\bmdefine\BGD{\Delta}
\bmdefine\BGF{\Phi}
\bmdefine\BGG{\Gamma}
\bmdefine\BGL{\Lambda}
\bmdefine\BGP{\Pi}
\bmdefine\BGT{\Theta}
\bmdefine\BGS{\Sigma}
\bmdefine\BGU{\Upsilon}
\bmdefine\BGO{\Omega}
\bmdefine\BGX{\Xi}
\bmdefine\BGY{\Psi}

\bmdefine\BCA{{\mathcal A}}
\bmdefine\BCB{{\mathcal B}}
\bmdefine\BCC{{\mathcal C}}
\bmdefine\BCD{{\mathcal D}}
\bmdefine\BCE{{\mathcal E}}
\bmdefine\BCF{{\mathcal F}}
\bmdefine\BCG{{\mathcal G}}
\bmdefine\BCH{{\mathcal H}}
\bmdefine\BCI{{\mathcal I}}
\bmdefine\BCJ{{\mathcal J}}
\bmdefine\BCK{{\mathcal K}}
\bmdefine\BCL{{\mathcal L}}
\bmdefine\BCM{{\mathcal M}}
\bmdefine\BCN{{\mathcal N}}
\bmdefine\BCO{{\mathcal O}}
\bmdefine\BCP{{\mathcal P}}
\bmdefine\BCQ{{\mathcal Q}}
\bmdefine\BCR{{\mathcal R}}
\bmdefine\BCS{{\mathcal S}}
\bmdefine\BCT{{\mathcal T}}
\bmdefine\BCU{{\mathcal U}}
\bmdefine\BCV{{\mathcal V}}
\bmdefine\BCW{{\mathcal W}}
\bmdefine\BCX{{\mathcal X}}
\bmdefine\BCY{{\mathcal Y}}
\bmdefine\BCZ{{\mathcal Z}}

\bmdefine\Bzr{ 0}
\bmdefine\Ba{ a}
\bmdefine\Bb{ b}
\bmdefine\Bc{ c}
\bmdefine\Bd{ d}
\bmdefine\Be{ e}
\bmdefine\Bf{ f}
\bmdefine\Bg{ g}
\bmdefine\Bh{ h}
\bmdefine\Bi{ i}
\bmdefine\Bj{ j}
\bmdefine\Bk{ k}
\bmdefine\Bl{ l}
\bmdefine\Bm{ m}
\bmdefine\Bn{ n}
\bmdefine\Bo{ o}
\bmdefine\Bp{ p}
\bmdefine\Bq{ q}
\bmdefine\Br{ r}
\bmdefine\Bs{ s}
\bmdefine\Bt{ t}
\bmdefine\Bu{ u}
\bmdefine\Bv{ v}
\bmdefine\Bw{ w}
\bmdefine\Bx{ x}
\bmdefine\By{ y}
\bmdefine\Bz{ z}
\bmdefine\BA{ A}
\bmdefine\BB{ B}
\bmdefine\BC{ C}
\bmdefine\BD{ D}
\bmdefine\BE{ E}
\bmdefine\BF{ F}
\bmdefine\BG{ G}
\bmdefine\BH{ H}
\bmdefine\BI{ I}
\bmdefine\BJ{ J}
\bmdefine\BK{ K}
\bmdefine\BL{ L}
\bmdefine\BM{ M}
\bmdefine\BN{ N}
\bmdefine\BO{ O}
\bmdefine\BP{ P}
\bmdefine\BQ{ Q}
\bmdefine\BR{ R}
\bmdefine\BS{ S}
\bmdefine\BT{ T}
\bmdefine\BU{ U}
\bmdefine\BV{ V}
\bmdefine\BW{ W}
\bmdefine\BX{ X}
\bmdefine\BY{ Y}
\bmdefine\BZ{ Z}

\begin{document}
\title{On the extreme rays of the cone of $3\times 3$ quasiconvex quadratic forms: Extremal determinants vs extremal and polyconvex forms}
\date{}
\author{Davit Harutyunyan\thanks{University of California Santa Barbara, harutyunyan@math.ucsb.edu}
and Narek Hovsepyan\thanks{Temple University, narek.hovsepyan@temple.edu}}

\maketitle
\begin{abstract}
This work is concerned with the study of the extreme rays of the convex cone of $3\times 3$ quasiconvex quadratic forms (denoted by ${\cal C}_3$). We characterize quadratic forms $f\in {\cal C}_3,$ the determinant of the acoustic tensor of which is an extremal polynomial, and conjecture/discuss about other cases. We prove that in the case when the determinant of the acoustic tensor of a form $f\in {\cal C}_3$ is an extremal polynomial other than a perfect square, then the form must itself be an extreme ray of ${\cal C}_3;$ when the determinant is a perfect square, then the form is either an extreme ray of ${\cal C}_3$ or polyconvex; and finally, when the determinant is identically zero, then the form $f$ must be polyconvex. The zero determinant case plays an important role in the proofs of the other two cases. We also
make a conjecture on the extreme rays of ${\cal C}_3,$ and discuss about weak and strong extremals of ${\cal C}_d$ for $d\geq 3,$ where it turns out that several properties of ${\cal C}_3$ do not hold for ${\cal C}_d$ for $d>3,$ and thus case $d=3$ is special. These results recover all previously known results (to our best knowledge) on examples of extreme points of ${\cal C}_3$ that were proved to be such. Our results also improve the ones proven by the first author and Milton [Comm. Pure Appl. Math., Vol. 70, Iss. 11, Nov. 2017, pp. 2164-2190] on weak extremals in ${\cal C}_3$ (or extremals in the sense of Milton) introduced in [Comm. Pure Appl. Math., Vol. XLIII, 63-125 (1990)].

In the language of positive biquadratic forms, quasiconvex quadratic forms correspond to nonnegative biquadratic forms and the results read as follows: If the determinant of the $\By$ (or $\Bx$) matrix of a $3\times 3$ nonnegative biquadratic form in $\Bx,\By \in\mathbb R^3$ is an extremal polynomial that is not a perfect square, then the form must be an extreme ray of the convex cone of $3\times 3$ nonnegative biquadratic forms $({\cal C}_3);$ if the determinant is identically zero, then the form must be a sum of squares; if the determinant is a nonzero perfect square, then the form is either an extreme ray of ${\cal C}_3,$ or is a sum of squares.

The proofs are all established by means of several classical results from linear algebra, convex analysis (geometry), real algebraic geometry, and the calculus of variations.

\end{abstract}

\textbf{Keywords:}\ \  Quasiconvex quadratic forms, positive biquadratic forms, sums of squares, polyconvexity, rank-one convexity.

\vspace{0.5cm}

\textbf{Mathematics Subject Classification:}\ \ 12D15, 12E10, 15A63, 49J40, 70G75, 74B05, 74B20,

\section{Introduction}
\setcounter{equation}{0}
\label{sec:1}

Let us point out from the onset that as we are applied mathematicians, the paper is written in the applied mathematics/calculus of variations language. However, the subject is in the intersection of the fields of \textbf{applied mathematics/calculus of variations} and \textbf{real algebraic geometry/convex geometry,} thus we have drawn some appropriate links between those two fields of mathematics in terms of language and results that we can understand.\\
 \textbf{Quasicovex quadratic forms and sums of squares: From applied mathematics to real algebraic geometry.} Quasiconvexity is a central subject in the calculus of variations and in applied mathematics. It was introduced by Morrey in 1952 [\ref{bib:Morrey.1},\ref{bib:Morrey.2}] and has several equivalent definitions, among which the simplest looking one is as follows [\ref{bib:Dacorogna}]: \textit{Let $n,N\in \mathbb N,$ and let the function $f\colon \mathbb R^{N\times n}\to\mathbb R$ be Borel measurable and locally bounded. Then $f$ is said to be quasiconvex, if}
\begin{equation}
\label{1.1}
f(\bm{\xi})\leq \int_{[0,1]^n}f(\bm{\xi}+\nabla\varphi(x))dx,
\end{equation}
\textit{for all matrices $\bm{\xi}\in\mathbb R^{N\times n}$ and all functions $\varphi\in W_0^{1,\infty}([0,1]^n,\mathbb R^N). $}
Under some appropriate growth conditions and some continuity conditions on the Lagrangian $f,$ it is known that quasiconvexity of $f$ in the gradient variable is equivalent to the fact that the energy functional
$$E(\By)=\int_{\Omega} f(x,\By(x),\nabla\By)dx$$
is weakly lower semicontinuous in an appropriate Sobolev space [\ref{bib:Morrey.1},\ref{bib:Ball},\ref{bib:Fon.Mue.},\ref{bib:Dacorogna},\ref{bib:Ben.Kru.}]; the weak lower semicontinuity of the energy $E$ in turn implies the existence of global minimizers for $E$ in the Sobolev space under consideration. The rank-one convexity condition, known to be a weaker than the quasiconvexity condition [\ref{bib:Sverak.2},\ref{bib:Sverak.1}], occurs when considering the second variation of the energy functional $E(\By).$ It reads as follows: \textit{Let $n,N\in \mathbb N$ and let $f\colon \mathbb R^{N\times n}\to\mathbb R.$ Then $f$ is said to be rank-one-convex, if
\begin{equation}
\label{1.2}
f(\lambda\BA+(1-\lambda)\BB)\leq \lambda f(\BA)+(1-\lambda)f(\BB),
\end{equation}
for all $\lambda\in [0,1]$ and $ \BA,\BB\in \mathbb R^{N\times n}$ such that $\rank(\BA-\BB)\leq 1.$} In linear elasticity a necessary condition for a body containing a linearly elastic homogeneous material with elasticity tensor $\BC$ to be stable, when the displacement is fixed at the boundary, is the rank-one convexity condition. In elasticity, when the material phase separates the displacement field (with no cracking), the displacement must still be continuous across the phase boundaries. Such phase separation is most easily seen in shape memory materials such as Nitinol. A simple geometry for the phase separated material is a laminate of the phases, and the continuity of the displacement field forces the difference of the displacement gradient in one phase minus the displacement field in the second phase to be a rank-one tensor. Thus to avoid this layering transformation the energy $f$ as a function of the displacement gradient must be rank one convex. More generally, to avoid separation at the microscale into other geometries of possibly lower energy (with affine boundary conditions on the displacement $\Bu$ at the boundary $\partial\Omega$ of the body) the energy $f(\nabla{\Bu}(\Bx))$ has to be a quasiconvex
function of $\nabla{\Bu(\Bx)}$ [\ref{bib:Ball},\ref{bib:Bal.Jam.}].

It is known that in the case when $f$ is a quadratic form, it is quasiconvex if and only if it is rank-one convex [\ref{bib:VanHove.1},\ref{bib:VanHove.2},\ref{bib:Dacorogna}],
which reduces to the so-called Legendre-Hadamard condition:
 \begin{equation}
\label{1.3}
f(\Bx\otimes\By)\geq 0, \quad \text{for all} \quad \Bx \in \mathbb R^N, \By\in\mathbb R^n,
\end{equation}
where $\Bx\otimes\By$ is the tensor product of the vectors $\Bx$ and $\By$ with $(\Bx\otimes\By)_{ij}=x_iy_j,$ for $1\leq i\leq N, 1\leq j\leq n.$
It is then clear that \textit{quasiconvex quadratic forms in applied mathematics correspond to nonnegative biquadratic forms in real algebraic geometry.} 
Let ${\cal C}_{N,n}$ denote the convex cone of $N\times n$ quasiconvex quadratic forms, where we set ${\cal C}_{n}={\cal C}_{n,n}.$ Another convexity condition in the calculus of variations is the polyconvexity condition introduced by Ball [\ref{bib:Ball}], which is known to be an intermediate condition between the standard convexity and quasiconvexity. \textit{A function $f\colon\mathbb R^{N\times n}\to\mathbb R$ is called polyconvex, if there exists a convex function $g\colon \mathbb R^K\to\mathbb R$ such that $f(\bm{\xi})=g(M_1,\dots,M_K)$ where $M_i$ are all the minors (including the first order ones) of the matrix $\bm{\xi}\in\mathbb R^{N\times n}$.} Terpstra [\ref{bib:Terpstra}] proved that in the special case when $f$ is a quadratic form, then $f$ is polyconvex if and only if it can be written as a convex quadratic form plus a linear combination of the second order minors of $\bm{\xi},$ see also [\ref{bib:Dacorogna}]. This means that \textit{polyconvex quadratic forms in applied mathematics correspond to biquadratic forms that are sums of squares in real algebraic geometry.} A characterization of symmetric polyconvexity has been recently given in [\ref{bib:Bou.Kre.Sch.}].  Also, a characterization of rank-one (quasiconvex) quadratic forms depending only on the strain is given by Zhang [\ref{bib:Zhang.1}] using Morse index. Ball showed that in the case $N=n,$ the determinant of the gradient function $\nabla \By$ is a Null-Lagrangian, and one has weak convergence of determinants $\det(\nabla\By_m)$ under the weak convergence of the fields $\{\By_m\}$ in a Sobolev space $W^{1,p} (n<p<\infty),$ thus the same classical theory of existence of global minimizers for convex Lagrangians goes through for polyconvex Lagrangians $f$ too [\ref{bib:Ball}]. There is no known algorithm that checks (analytically or even numerically) if the given function is quasiconvex or not, and it is surprisingly very complex even for simple functions $f,$ while checking the polyconvexity of a function $f$ can be straightforward in many cases. This makes polyconvexity much easier to deal with. The present work continues the line of studying extreme rays and the so-called Milton extremals (or simply weak extremals) of ${\cal C}_3,$ initiated in [\ref{bib:Har.Mil.1}] and further developed in [\ref{bib:Har.Mil.2},\ref{bib:Har.Mil.3},\ref{bib:Harutyunyan}]. Namely, we study the elements of ${\cal C}_3$ that have an extremal acoustic tensor determinant as a polynomial, and characterize them. For the convenience of the reader, we next present definitions of weak and strong extremals (extreme rays) of ${\cal C}_3,$ (for the definition of the acoustic tensor see the paragraph right before Thorem~\ref{th:1.1}). 

\begin{definition}
\label{def:1.1}
A quasiconvex quadratic form $f(\bm{\xi})\colon\mathbb R^{N\times n}\to\mathbb R$ ($f\in {\cal C}_{N,n}$) is called 
\begin{itemize}
\item[(i)] A weak (or Milton) extremal, if one can not subtract a convex form from it, other then a multiple of itself, preserving the quasiconvexity of $f.$ 
\item[(i)] An extreme ray of ${\cal C}_{N,n}$ (or a strong extremal), if one can not subtract a quasiconvex form from it, other then a multiple of itself, preserving the quasiconvexity of $f.$ 
\end{itemize}
\end{definition}

It is not difficult to prove that even the notion of weak extremality in ${\cal C}_{N,n}$ has the Krein-Milman property [\ref{bib:Milton.3}]. Extremals (weak or strong) are known to play an important role in the theory of composites as suggested by the work [\ref{bib:Milton.3}], especially when bounding effective properties of composites (such as shear or bulk moduli in elasticity for instance), in particular, the simplest forms of extremals that are the $2\times 2$ minors of $\bm{\xi}$ (which are also Null-Lagrangians), are the basis of the so-called translation method of Murat and Tartar [\ref{bib:Mur.Tar.},\ref{bib:Tartar}] or Cherkaev and Gibiansky [\ref{bib:Che.Gib.1}], see also the works [\ref{bib:Tartar},\ref{bib:Che.Gib.2},\ref{bib:All.Koh.},\ref{bib:Kan.Kim.Mil.},\ref{bib:Koh.Lip.},\ref{bib:Kan.Mil.Wan.},\ref{bib:Mil.Ngu.}] and the books [\ref{bib:Milton.1},\ref{bib:Cherkaev}]. Special forms of extremals have been used by Kang and Milton in [\ref{bib:Kan.Mil.}]
to prove bounds on the volume fractions of two materials in a three dimensional body from boundary measurements.
Extremal quasiconvex forms are also the best choice of quasiconvex functions for obtaining series expansions for effective tensors that have an extended domain of convergence, and thus analyticity properties as a function of the component moduli on this domain
(see section 14.8, and page 373 of section 18.2 of [\ref{bib:Milton.2}]).

It is easy to see that any nontrivial extremal quadratic form (different from the square of a linear form or linear combination of $2\times 2$ minors) is automatically an example of a quasiconvex quadratic form that is not polyconvex. Note, as proven by Terpstra [\ref{bib:Terpstra}], that a quadratic form is polyconvex if and only if it is the sum of a convex form and a linear combination of second order minors of the matrix $\bm{\xi}.$ It was an open question in the applied mathematics community to find an explicit example of a quadratic form that is not polyconvex, until Serre provided one [\ref{bib:Serre}] in 1981. Surprisingly such an example was already provided in linear algebra/real algebraic geometry community by Choi [\ref{bib:Choi}] six years earlier in 1975, which had not been known to the applied mathematics communities until very recent times (we believe until the year 2019). Two years later Choi and Lam provided another, even more beautiful explicit example of such a form in [\ref{bib:Cho.Lam.}]:
 \begin{equation}
\label{1.4}
f(\bm{\xi})=\xi_{11}^2+\xi_{22}^2+\xi_{33}^2+\xi_{12}^2+\xi_{23}^2+\xi_{31}^2-2(\xi_{11}\xi_{22}+\xi_{22}\xi_{33}+\xi_{33}\xi_{11}),
\end{equation}
where they prove that the new example is in fact an extreme ray (the first such explicit example) of ${\cal C}_3;$ see also [\ref{bib:Cho}]. In fact it is an open question whether weak and strong extremals of ${\cal C}_3$ are the same, while for ${\cal C}_d,$ $d\geq 4$ they are different, see next section. The first author and Milton came up with the Choi-Lam example later in [\ref{bib:Har.Mil.1}] being unaware of it (as the applied mathematics community was unaware of it) due to the lack of communication between the two communities/fields. Nonnegative biquadratic forms have been a central subject of interest in the real algebraic geometry community, such as extreme points of the convex cone ${\cal C}_d$ [\ref{bib:Choi},\ref{bib:Cho.Lam.}], separability and inseparability of positive linear maps [\ref{bib:Stormer},\ref{bib:Hou.Li.Poo.Qi.Sze.},\ref{bib:Li.Wu.}], maximal possible number of their zeros and connections with extremality [\ref{bib:Quarez.1},\ref{bib:Buc.Siv.}]. In particular the problem of expressing a nonnegative homogeneous polynomial as a sum of squares is very famous in real algebraic geometry [\ref{bib:Hilbert},\ref{bib:Artin},\ref{bib:Ble.Smi.Vel.},\ref{bib:Reznick},\ref{bib:Ble.Sin.Smi.Vel.},\ref{bib:Blekherman.1},\ref{bib:Blekherman.2},\ref{bib:Scheiderer}]. In 1888 Hilbert raised the question of whether any nonnegative polynomial over reals can be expressed as a sum of squares of rational functions, which was solved in the affirmative by Artin [\ref{bib:Artin}]. For the problem of sums of squares of polynomials we refer to the recent surveys by Blekherman and coauthors [\ref{bib:Blekherman.2},\ref{bib:Ble.Sin.Smi.Vel.}]. Another very important related problem in applied mathematics, concerning sixth order homogeneous polynomials in three variables and determinants of $\By-$matrices of quadratic forms is, whether or not any such polynomial, in particular the well known Robinson's polynomial, is a determinant of a $\By-$matrix, and it is open as well [\ref{bib:Reznick},\ref{bib:Quarez.1}]. In [\ref{bib:Buc.Siv.}] the authors construct the first examples of nonnegative biquadratic forms with a tensor in $(\mathbb R^3)^4,$ that have maximal number of nontrivial zeros, namely ten of them. Note that by our result in Theorem~\ref{th:2.1}, the latter are extreme rays of ${\cal C}_3,$ as their $\By-$matrix determinants are scalar multiples of the generalized Robinson's polynomial [\ref{bib:Buc.Siv.},\ref{bib:Reznick}], which is an extremal polynomial. 

\textit{Recall that an $2n-$homogeneous polynomial $P(\Bx)$ in the variable $\Bx=(x_1,x_2,\dots,x_m)$ is said to be an extremal polynomial, if $\mathrm{deg}(P)=2n,$ $P(\Bx)\geq 0$ for all $\Bx\in \mathbb R^m,$ and $P(\Bx)$ can not be split into the sum of two linearly independent polynomials $P_1$ and $P_2$ having the same properties.}

Some of the above results were used in [\ref{bib:Har.Mil.3}] to come up with a sufficient condition for a form 
 $f(\bm{\xi})=\bm{\xi}\BC\bm{\xi}^T\in {\cal C}_d$ to be a weak extremal, where
$\bm{\xi}\in \mathbb R^{d\times d}$ and $\BC\in (\mathbb R^d)^4.$ Namely, let a rank-one matrix
$\bm{\xi}\in \mathbb R^{d\times d}$ be given as $\bm{\xi}=\Bx\otimes\By,$ where $\Bx,\By\in\mathbb R^d,$ $d\geq 3.$ Then one can write
$f(\bm{\xi})=f(\Bx\otimes\By)=\Bx T(\By)\Bx^T,$ where $T(\By)$ is a $d\times d$ matrix, called the acoustic tensor (or just $\By-$matrix) of $f,$ with entries being quadratic forms in $\By.$ The following results have been proven in [\ref{bib:Har.Mil.3}] (we combine Theorems 3.4-3.7 in one).

\begin{theorem}
\label{th:1.2}
Let the quadratic form $f(\bm{\xi})=\bm{\xi}\BC\bm{\xi}^T,$ where $\bm{\xi}\in \mathbb R^{d\times d}$ and $\BC\in (\mathbb R^d)^4,$ $d\geq 3$ be quasiconvex. Then
\begin{itemize}
\item[(i)] If the determinant $\det(T(\By))$ is an irreducible (over the reals) extremal polynomial, then the form $f$ is a weak extremal.
\item[(ii)] Assume $d=3.$ If the determinant $\det(T(\By))$ is an extremal that is not a perfect square, then $f$ is a weak extremal.
\item[(iii)] Assume $d=3.$ If $\det(T(\By))\equiv 0$ then the form $f$ is either a weak extremal or polyconvex.
\item[(iv)] Assume $d=3.$ If the determinant $\det(T(\By))$ is a perfect square (note that this automatically implies that it is an extremal polynomial as can be seen easily), then $f$ is either a weak extremal, polyconex, or the sum of a polyconvex and a weak extremal forms, where the extremal form has identically zero acoustic tensor determinant.
\end{itemize}
\end{theorem}
We improved the result in (ii) for forms having linear elastic orthotropic symmetry in [\ref{bib:Harutyunyan}], showing that in fact under the extremality and non-square condition on the determinant $\det(T(\By)),$ the form $f$ must in fact be an extreme ray of
${\cal C}_3.$ In the present manuscript we study forms $f\in {\cal C}_3$ in questions (ii)-(iv) for strong extremality, see Theorems 2.1-2.2 in the next section. We also conjecture about weak versus strong extremals of ${\cal C}_3,$ about extremality or the vanishing property of the acoustic tensor determinant versus weak or strong extremality or polyconvexity of $f(\bm{\xi})$ for $d=3$ and $d\geq 4,$ see next section.

\section{Main Results}
\setcounter{equation}{0}
\label{sec:2}
Let $d\in\mathbb N$, ($d\geq 3$) and $\BC=(C_{ijkl})\in (\mathbb R^{d})^4$ be a fourth order tensor with usual symmetries:
\begin{equation}
\label{2.1}
C_{ijkl}=C_{kjil}=C_{ilkj},\qquad 1\leq i,j,k,l\leq d.
\end{equation}
In what follows we will regard the matrix $\bm{\xi}=(\xi_{ij})\in\mathbb R^{d\times d}$ as a $d^2$-vector so that the quadratic form $f(\bm{\xi})$ will be given by
\begin{equation}
\label{2.2}
f(\bm{\xi})=\bm{\xi} \BC \bm{\xi}^T=\sum_{1\leq i,j,k,l\leq d}C_{ijkl}\xi_{ij}\xi_{kl},
\end{equation}
which will be applicable in the context of elasticity. As already noted, in the special case when $\bm{\xi}=\Bx\otimes\By$ is a rank-one matrix, where $\Bx,\By\in\mathbb R^d,$ the quadratic form $f$ reduces to
\begin{equation}
\label{2.3}
f(\Bx\otimes\By)=\Bx (\By\BC\By^T)\Bx^T=\Bx \BT(\By)\Bx^T,
\end{equation}
where $\BT(\By)=\By\BC\By^T\in \mathbb R^{d\times d}$ is the acoustic tensor (or simply the $\By-$matrix) of $f.$ Also, as mentioned above, it turns out that the determinant of $\BT(\By)$ tells quite a lot about the form $f,$ which is quite unexpected [\ref{bib:Har.Mil.3}]. We will focus on the case when $\det(\BT(\By))$ is an extremal polynomial. The following are the main results of the paper. The first theorem refers to the cases (ii) and (iv) in Thereom~\ref{th:1.2}.
\begin{theorem}
\label{th:2.1}
Let $f(\bm{\xi})=\bm{\xi} \BC \bm{\xi}^T\in {\cal C}_3,$ where $\bm{\xi}\in \mathbb R^{3\times 3}$ and $\BC\in(\mathbb R^3)^4$ is a fourth order tensor with usual symmetries as in (\ref{2.1}). Assume that the determinant of the $\By-$matrix of
$f(\Bx\otimes \By)$ is an extremal polynomial. Then one has the following:
\begin{itemize}
\item[1.] If $\det(\BT(\By))$ is not a perfect square, then $f$ must be an extreme ray of ${\cal C}_3.$
\item[2.] If $\det(\BT(\By))$ is a perfect square, then $f$ is either an extreme ray of ${\cal C}_3$ or polyconvex.
\end{itemize}
\end{theorem}

The next theorem refers to the case (iii) in Thereom~\ref{th:1.2}. It will also be a major factor in the proof of the main Theorem 2.1.

\begin{theorem}
\label{th:2.2}
Let $f(\bm{\xi})=\bm{\xi} \BC \bm{\xi}^T\in {\cal C}_3,$ where $\bm{\xi}\in \mathbb R^{3\times 3}$ and $\BC\in(\mathbb R^3)^4$ is a fourth order tensor with usual symmetries as in (\ref{2.1}). Assume that the determinant of the $\By-$matrix of $f(\Bx\otimes \By)$ is identically zero. Then $f$ must be a polyconvex form.
\end{theorem}

Several remarks are in order.

\begin{remark}[The case $d=3$]
\label{re:2.3}
Let $d=3.$ Note first that the Choi-Lam example in (\ref{1.4}) gives
$$\det{\BT(\By)}=y_1^4y_2^2+y_2^4y_3^2+y_3^4y_1^2-3y_1^2y_2^2y_3^2,$$
which is known to be an extremal polynomial; this falls into Theorem~2.1. An example of a polyconvex $f$ that has a perfect square or zero determinant would be $f(\bm{\xi})=\sum_{i=1}^3\xi_{ii}^2$ or $f(\bm{\xi})=\xi_{11}^2.$ However, we are not
aware of an example of an $f$ that is non-polyconvex, is an extreme ray of ${\cal C}_3$ such that $\det(\BT(\By))$ is a perfect square.
We believe that if $f\in {\cal C}_3$ with $\det{\BT(\By)}$ being a perfect square, then $f$ must in fact be polyconvex. However, at the moment we have no proof for the statement.
\end{remark}

\begin{remark}[The case $d\geq 4$]
\label{re:2.4}
Note that if one only assumes that $\det{\BT(\By)}$ is en extremal polynomial (not necessarily irreducible), then $f(\bm{\xi})$ has to be neither a weak extremal of ${\cal C}_d$ nor polyconvex. A counterexample would be
$$f(\bm{\xi})=\xi_{11}^2+\xi_{22}^2+\xi_{33}^2+\xi_{12}^2+\xi_{23}^2+\xi_{31}^2-2(\xi_{11}\xi_{22}+\xi_{22}\xi_{33}+\xi_{33}\xi_{11})\
+\sum_{k=3}^d\xi_{kk}^2,$$
which has the acoustic tensor determinant
$$\det{\BT(\By)}=(y_1^4y_2^2+y_2^4y_3^2+y_3^4y_1^2-3y_1^2y_2^2y_3^2)\prod_{k=3}^dy_{k}^2,$$
which is clearly an extremal polynomial. However, obviously $f$ is neither a weak extremal nor polyconvex.
\end{remark}

\begin{remark}[The case $d\geq 4$]
\label{re:2.5}
Another thing to note is that in the case $d\geq 4,$ one can put together two copies of the Choi-Lam form to achieve an example
$f\in {\cal C}_d$ that is a weak but not a strong extremal of ${\cal C}_d.$ Namely, it is easy to see that the form
\begin{align*}
f(\bm{\xi})&=\xi_{11}^2+\xi_{22}^2+\xi_{33}^2+\xi_{12}^2+\xi_{23}^2+\xi_{31}^2-2(\xi_{11}\xi_{22}+\xi_{22}\xi_{33}+\xi_{33}\xi_{11})\\
&+\xi_{22}^2+\xi_{33}^2+\xi_{44}^2+\xi_{23}^2+\xi_{34}^2+\xi_{41}^2-2(\xi_{22}\xi_{33}+\xi_{33}\xi_{44}+\xi_{44}\xi_{22})
\end{align*}
is a weak extremal of ${\cal C}_d.$ This implies that weak and strong extremals of ${\cal C}_d$ are in general different for $d\geq 4.$
\end{remark}

\begin{remark}[The case $d\geq 4$]
\label{re:2.6}
Taking again the Choi-Lam example $f(\bm{\xi})$ we have $\det{\BT(\By)}\equiv 0$ for $d\geq 4.$ This shows that Theorem~2.2 fails for $d\geq 4.$
\end{remark}

Finally, we make following conjecture.
\begin{conjecture}[The case $d=3$]
\label{con:2.7}
Any non-polyconvex weak extremal $f\in {\cal C}_3$ is an extreme ray of ${\cal C}_3.$ Moreover, if $f\in {\cal C}_3$ is a
non-polyconvex extreme ray of ${\cal C}_3,$ then $\det{\BT(\By)}$ is en extremal polynomial different from a perfect square.
\end{conjecture}
The motivation behind this conjecture is the yet unproven fact that any nonnegative sixth degree homogeneous polynomial $P(\By)$ in three variables 
($\By\in\mathbb R^3$) is necessarily the determinant of the acoustic tensor $T(\By)$ of an element $f\in {\cal C}_3,$ e.g.,
[\ref{bib:Buc.Siv.},\ref{bib:Reznick},\ref{bib:Quarez.1}]. A weaker statement, that every real multivariate polynomial has a symmetric determinantal representation is 
known to be true, and was recently proven by Helton, McCullough, and Vinnikov [\ref{bib:Hel.McC.Vin.}], see also [\ref{bib:Quarez.2},\ref{bib:Ste.Wel.}].

\section{Proof of Theorem~2.1}
\label{sec:3}
\setcounter{equation}{0}

\begin{proof}[Proof of Theorem~\ref{th:2.1}] We will be utilizing Theorem~2.2 in the proof here; the proof of which is postponed until Section~4.
We will be carrying out some steps applicable to both cases in Theorem~\ref{2.1}, and at the same time considering each case separately if necessary. Assume in contradiction that $f$ is not an extreme ray of ${\cal C}_3,$ thus there exists a form $f_1\in{\cal C}_3$ such that $f_1$ and $f$ are linearly independent satisfying the inequalities
\begin{equation}
\label{4.1}
0\leq f_1(\Bx\otimes\By)\leq f(\Bx\otimes\By),\quad\text{for all}\quad \Bx,\By\in\mathbb R^3.
\end{equation}
We will prove that in the first case this is not possible, while in the second case this leads to the conclusion that $f$ is polyconvex.
Denote $f(\Bx\otimes\By)=\Bx \BT(\By)\Bx^T,$ $f_1(\Bx\otimes\By)=\Bx \BT^1(\By)\Bx^T,$ $\BT(\By)=(t_{ij}(\By))_{i,j=1}^3,$ and $\BT^1(\By)=(t_{ij}^1(\By))_{i,j=1}^3.$ Consider the determinant $\det(\BT(\By)-\lambda \BT^1(\By))$ as a polynomial
in $\lambda\in\mathbb R:$
\begin{align}
\label{4.2}
P(\lambda)&=\det(\BT(\By)-\lambda \BT^1(\By))\\ \nonumber
&=\det(\BT(\By))-\lambda\sum_{i,j=1}^3t^1_{ij}(\By)\cof(\BT(\By))_{ij}+\lambda^2\sum_{i,j=1}^3t_{ij}(\By)\cof(\BT^1(\By))_{ij}-\lambda^3\det(\BT^1(\By)),
\end{align}
which will be a key factor in the analysis. The determinant above gives rise to the coefficients of $\lambda^k,$ for $k=0,1,2,3$ that are homogeneous polynomials of $\By$ of degree six, which turn out to have to satisfy certain monotonicity properties proven in
[Lemma~4.1, \ref{bib:Harutyunyan}] and given in the lemma below.
\begin{lemma}
\label{lem:4.1}
Let $n\in\mathbb N$ satisfy $n\geq 2$ and let $\BA,\BB\in \mathbb M_{sym}^{n\times n}$ be symmetric positive semi-definite matrices such that $\BA\geq \BB$ in the sense of quadratic forms. Then for any integers $1\leq k< m\leq n$ one has the inequality
\begin{equation}
\label{4.3}
\frac{1}{{n \choose m}}\sum_{M_m(\BB)}M_m(\BB)\cof_{\BA}(M_m(\BB))\leq \frac{1}{{n \choose k}}\sum_{M_k(\BB)}M_k(\BB)\cof_{\BA}(M_k(\BB)),
\end{equation}
where the number ${n \choose m}$ is the binomial coefficient, and the sum $\sum_{M_m(\BB)}$ is taken over all $m-$th order minors $M_m(\BB)$ of $\BB,$ and $\cof_{\BA}(M_m(\BB))$ denotes the cofactor of the minor in the matrix $\BA,$ obtained by choosing the same rows and columns as to get the minor $M_m(\BB)$ in $\BB.$
\end{lemma}
Due to (\ref{4.1}), we have $\BT(\By)\geq \BT^1(\By)$ for all $\By\in\mathbb R^3$ in the sense of quadratic forms, thus Lemma~\ref{lem:4.1} implies the inequalities
\begin{equation}
\label{4.4}
0\leq 3\det(\BT^1(\By))\leq \sum_{i,j=1}^3t_{ij}(\By)\cof(\BT^1(\By))_{ij}\leq \sum_{i,j=1}^3t^1_{ij}(\By)\cof(\BT(\By))_{ij}\leq 3\det(\BT(\By)),\ \ \  \By\in\mathbb R^3.
\end{equation}
Hence the polynomials $3\det(\BT^1(\By))$, $\sum_{i,j=1}^3t_{ij}(\By)\cof(\BT^1(\By))_{ij},$ $\sum_{i,j=1}^3t^1_{ij}(\By)\cof(\BT(\By))_{ij},$ being in between zero and the extremal polynomial $3\det(\BT(\By))$ must be scalar multiples of $\det(\BT(\By)),$ i.e., we have
\begin{align}
\label{4.5}
\det(\BT^1(\By))&=\alpha \det(\BT(\By)),\\ \nonumber
\sum_{i,j=1}^3t_{ij}(\By)\cof(\BT^1(\By))_{ij}&=\beta\det(\BT(\By)),\\ \nonumber
\sum_{i,j=1}^3t^1_{ij}(\By)\cof(\BT(\By))_{ij}&=\gamma \det(\BT(\By)),\\ \nonumber
\text{for some}\quad \alpha, \beta, \gamma &\geq 0.
\end{align}
Consequently we get from (\ref{4.2}) and (\ref{4.5}) the key identity
\begin{equation}
\label{4.6}
\det(\BT(\By)-\lambda \BT^1(\By))=(1-\gamma\lambda+\beta\lambda^2-\alpha\lambda^3)\det(\BT(\By))=\varphi(\lambda)\det(\BT(\By)),\ \ \
\By\in \mathbb R^3,\lambda\in\mathbb R.
\end{equation}
In the next step we note that the polynomial $\varphi$ does not have roots in $(-\infty,1)$, more precisely
\begin{equation}
\label{4.7}
\varphi(\lambda)>0,\quad\text{ for}\quad \lambda\in(-\infty,1).
\end{equation}
Indeed, for $\lambda\leq 0$ we have by the conditions $\alpha,\beta,\gamma \geq 0$ that $\varphi(\lambda)\geq 1.$
Choosing a point $\By^0\in\mathbb R^3$ such that $\det(\BT(\By^0))>0$, we have for any $\lambda\in(0,1)$ by Lemma~\ref{lem:4.1} that
\begin{align*}
\varphi(\lambda)&=\frac{1}{\det(\BT(\By^0))}\det(\BT(\By^0)-\lambda \BT^1(\By^0))\\
&=\frac{1}{\det(\BT(\By^0))}\det[(1-\lambda)\BT(\By^0)+\lambda(\BT(\By^0)-\BT^1(\By^0))]\\
&\geq (1-\lambda)^3,
\end{align*}
as $\BT(\By^0)\geq \BT^1(\By^0)$ in the sense of quadratic forms. Note also that the equality $\alpha=0$ is impossible as it would mean by (\ref{4.5}) that
$$\det(\BT^1(\By))=\alpha \det(\BT(\By))=0,\qquad \By\in\mathbb R^3,$$
i.e., the quasiconvex form $f^1$ has an identically zero acoustic tensor determinant, thus by Theorem~\ref{th:2.2} it must be polyconvex. Invoking again the characterization theorem for polyconvex quadratic forms by Terpstra [\ref{bib:Terpstra}], we infer that $f^1$ is a sum of squares (at least one), which means by (\ref{4.1}) that in fact one can subtract a perfect square form $f$ still preserving the quasiconvexity of $f,$ i.e., $f$ is not a weak extremal, which contradicts part (ii) of Theorem~\ref{th:1.2}. Consequently we must have $\alpha>0$ and
$\det(\BT^1(\By))=\alpha \det(\BT(\By))>0$ whenever $\det(\BT(\By))>0.$ Also it is important to note that $\varphi$ is necessarily a third degree polynomial. Choose again $\By_0\in\mathbb R^3$ (as above) such that $\det(\BT(\By_0))>0.$ Hence setting
$\BA=\BT(\By_0)$ and $\BB=\BT^1(\By_0)$ we have $\det(\BA)\geq \det(\BB)>0,$ where $\BA,\BB\in\mathbb R^{3\times 3}$ are symmetric positive definite matrices, thus the square root $\BB^{1/2}$ and the inverse $\BB^{-1/2}$ exist and are symmetric. Next we have from (\ref{4.2}) and (\ref{4.6}),
\begin{align*}
P(\lambda)&=\det(\BA-\lambda\BB)\\
&=\det (\BB^{1/2}(\BB^{-1/2}\BA\BB^{-1/2}-\lambda\BI)\BB^{1/2})\\
&=\det(\BB)\det(\BB^{-1/2}\BA\BB^{-1/2}-\lambda\BI)\\
&=\det(\BA)\varphi(\lambda),
\end{align*}
which gives
$$\varphi(\lambda)=\alpha\cdot\det(\BB^{-1/2}\BA\BB^{-1/2}-\lambda\BI),$$
thus the roots of $\varphi$ are real as $\varphi$ is a scalar multiple of the characteristic polynomial of the symmetric matrix
$\BB^{-1/2}\BA\BB^{-1/2}.$ On the other hand (\ref{4.7}) implies that all three roots of $\varphi$ belong to the interval $[1,\infty).$ Denoting them by $1\leq \lambda_1\leq \lambda_2\leq \lambda_3$ we have
$$\varphi(t)=(1-\gamma\lambda+\beta\lambda^2-\alpha\lambda^3)=-\alpha(\lambda-\lambda_1)(\lambda-\lambda_2)(\lambda-\lambda_3),$$
and we have by Vieta's theorem the formulae
\begin{equation}
\label{4.8}
\alpha=\frac{1}{\lambda_1\lambda_2\lambda_3},\quad\beta=\frac{\lambda_1+\lambda_2+\lambda_3}{\lambda_1\lambda_2\lambda_3},
\quad \gamma=\frac{\lambda_1\lambda_2+\lambda_2\lambda_3+\lambda_3\lambda_1}{\lambda_1\lambda_2\lambda_3},
\end{equation}
which will be utilized in the next steps. Next introduce the biquadratic form
$$g(\Bx\otimes\By)=f(\Bx\otimes\By)-\lambda_1f^1(\Bx\otimes\By).$$
The strategy from here on is to either prove that $g$ is identically zero, which will imply that $f_1$ is a multiple of $f$, or otherwise to arrive at a contradiction in the case when $\det(\BT(\By))$ is not a perfect square, or prove that $f$ is polyconvex in the case when $\det(\BT(\By))$ is a perfect square. We have from (\ref{4.6}) that $\det(\BS(\By))\equiv 0$ for $\By\in\mathbb R^3,$ where $\BS(\By)=(s_{ij}(\By))_{1\leq i,j\leq 3}$ is the acoustic tensor of $g,$ i.e. $g(\Bx\otimes\By)=\Bx \BS(\By)\Bx^T$. Note that as $f$ and $f_1$ are linearly independent, then the form $g$ is not identically zero. Next we aim to prove that the diagonal entries of the cofactor matrix $\cof(\BS(\By))$ are nonnegative. If they all vanish identically, then there is nothing to prove. Assume for instance $\cof(\BS(\By))_{33}$ does not vanish identically, then the set $\{\By\in\mathbb R^3 : \cof(\BS(\By))_{33}=0\}$ is a null set, thus because $\det(\BS(\By))\equiv 0,$ the last row of $\BS$ must be a linear combination of the first two for a.e. $\By\in\mathbb R^3,$ thus we obtain the form
\begin{equation}
\label{4.9}
\BS(\By)=
\begin{bmatrix}[1.5]
s_{11} & s_{12} & rs_{11}+qs_{12}\\
s_{12} & s_{22} & rs_{12}+qs_{22}\\
rs_{11}+qs_{12} & rs_{12}+qs_{22} & r^2s_{11}+q^2s_{22}+2rqs_{12}
\end{bmatrix},
\end{equation}
where the linear combination coefficients $r$ and $q$ are rational functions given by
\begin{equation}
\label{4.10}
r(\By)=\frac{\cof(\BS(\By))_{13}}{\cof(\BS(\By))_{33}},\qquad q(\By)=-\frac{\cof(\BS(\By))_{23}}{\cof(\BS(\By))_{33}}.
\end{equation}
Note that (\ref{4.9}) also yields the form of the cofactor matrix $\cof(\BS):$
\begin{equation}
\label{4.11}
 \cof(\BS)=
\begin{bmatrix}[1.5]
r^2\cdot\cof(\BS)_{33} & rq\cdot\cof(\BS)_{33} & -r\cdot\cof(\BS)_{33} \\
rq\cdot\cof(\BS)_{33} & q^2\cdot\cof(\BS)_{33} & -q\cdot\cof(\BS)_{33} \\
-r\cdot\cof(\BS)_{33} & -q\cdot\cof(\BS)_{33} & \cof(\BS)_{33}
\end{bmatrix}.
\end{equation}
Now using the equality $f=g+\lambda_1f_1$ and formula (\ref{4.2}) we get
\begin{align*}
\det(\BT(\By))&=\det(\BS(\By)+\lambda_1\BT^1(\By))\\
&=\det(\BS(\By))+\lambda_1\sum_{i,j=1}^3t^1_{ij}(\By)\cof(\BS(\By))_{ij}+\lambda_1^2\sum_{i,j=1}^3s_{ij}(\By)\cof(\BT^1(\By))_{ij}
+\lambda_1^3\det(\BT^1(\By))\\
&=\lambda_1\sum_{i,j=1}^3t^1_{ij}(\By)\cof(\BS(\By))_{ij}+\lambda_1^2\sum_{i,j=1}^3s_{ij}(\By)\cof(\BT^1(\By))_{ij}+\lambda_1^3\det(\BT^1(\By)),
\end{align*}
hence owing to the first equality in (\ref{4.5}) we obtain
\begin{equation}
\label{4.12}
\sum_{i,j=1}^3t^1_{ij}(\By)\cof(\BS(\By))_{ij}=\left(\frac{1}{\lambda_1}-\lambda_1^2\alpha\right)\det(\BT(\By))
-\lambda_1\sum_{i,j=1}^3s_{ij}(\By)\cof(\BT^1(\By))_{ij}.
\end{equation}
We have further utilizing the first two identities in (\ref{4.5}), that
\begin{align}
\label{4.13}
\sum_{i,j=1}^3s_{ij}(\By)\cof(\BT^1(\By))_{ij}&=\sum_{i,j=1}^3(t_{ij}-\lambda_1t_{ij}^1)(\By)\cof(\BT^1(\By))_{ij}\\ \nonumber
&=(\beta-3\lambda_1\alpha)\det(\BT(\By)),
\end{align}
thus owing back to (\ref{4.12}) we obtain
\begin{equation}
\label{4.14}
\sum_{i,j=1}^3t^1_{ij}(\By)\cof(\BS(\By))_{ij}=\left(\frac{1}{\lambda_1}+2\lambda_1^2\alpha-\lambda_1\beta\right)\det(\BT(\By)).
\end{equation}
Consequently recalling (\ref{4.8}) and (\ref{4.11}) we get from (\ref{4.14}) after some simple algebra,
\begin{equation}
\label{4.15}
\cof(\BS(\By))_{33}\cdot[(-r,-q,1)\BT^1(\By)(-r,-q,1)^T]=\frac{(\lambda_2-\lambda_1)(\lambda_3-\lambda_1)}{\lambda_1\lambda_2\lambda_3}\det(\BT(\By)).
\end{equation}
The last equality suggests considering the following cases separately.\\
\textbf{Case 1: $\lambda_1<\lambda_2.$}\\
\textbf{Case 2: $\lambda_1=\lambda_2<\lambda_3.$}\\
\textbf{Case 3: $\lambda_1=\lambda_2=\lambda_3.$}\\
\textbf{Case 1.} In this case by the fact that $\det(\BT(\By))>0$ a.e. in $\mathbb R^3$ and by the positive semi-definiteness of $\BT^1(\By),$ the equality (\ref{4.15}) immediately implies that $\cof(\BS(\By))_{33}\geq 0$ for all $\By\in\mathbb R^3.$ Therefore we have $\cof(\BS(\By))_{ii}\geq 0$ for all $\By\in\mathbb R^3,$ $i=1,2,3.$ Consider next the following two cases.\\
\textbf{Case 1a: One of the diagonal entries of $\BS(\By)$ is definite.}\\
\textbf{Case 1b: All of the diagonal entries of $\BS(\By)$ are indefinite.}\\
\textbf{Case 1a.} In this case if say $s_{11}(\By)$ is positive semidefinite, then we get by Silvester's criterion that $\BS(\By)$ is positive semidefinite, thus the form $g$ will become a quasiconvex quadratic form that has zero acoustic tensor determinant, thus by Theorem~\ref{2.2} it must be polyconvex. As $g$ is not identically zero, it must be a sum of squares, containing at least one square, thus the condition $f=\lambda_1f^1+g$ will imply that $f$ is not a weak extremal, which contradicts Theorem~\ref{th:1.2} in the case when $\det(\BT(\By))$ is not a perfect square. Considering the case when $\det(\BT(\By))$ is a perfect square, note that part (iv) of Theorem~\ref{th:1.2} together with Theorem~2.2 imply that $f$ has to be either a weak extremal or polyconvex. The case of a weak extremal is again ruled out by the equality $f=\lambda_1f^1+g$ as $f_1$ is quasiconvex and $g$ is nonzero and polyconvex. Thus we conclude that $f$ is polyconvex. Now, in the case when all of $s_{ii}(\By)$ are negative semidefinite, then again by Silvester's criterion we have that $\BS(\By)$ must be negative semidefinite. Recall next the following classical linear algebra (convex analysis) theorem. It has to do with the fact that the convex cone of all $n\times n$ positive semidefinite symmetric matrices is self-dual.
\begin{theorem}
\label{th:4.2}
Let $n\in\mathbb N$ and let $\BA=(a_{ij}),\BB=(b_{ij})\in\mathbb R^{n\times n}$ be symmetric positive semidefinite matrices. Then the inner product of $\BA$ and $\BB$ is nonnegative:
$$\BA\colon\BB=\sum_{i,j=1}^na_{ij}b_{ij}\geq 0.$$
\end{theorem}
Note next that (\ref{4.13}) implies the equality
\begin{equation}
\label{4.16}
\sum_{i,j=1}^3s_{ij}(\By)\cof(\BT^1(\By))_{ij}=\frac{\lambda_2+\lambda_3-2\lambda_1}{\lambda_1\lambda_2\lambda_3}\det(\BT(\By)),
\end{equation}
where the right hand side is strictly positive a.e. in $\mathbb R^3,$ while the left hand side is nonpositive due to Thereom~\ref{th:4.2} and the fact that $\BS$ is negative semidefinite and $\cof(\BT^1)$ is positive semidefinite. This gives a contradiction.\\
\textbf{Case 1b.} We start by recalling the following theorem by Marcellini [Corollary~1, \ref{bib:Marcellini}].

\begin{theorem}[Marcellini]
\label{th:4.3}
Let $Q_1$ and $Q_2$ be two quadratic forms in $\mathbb R^n,$ with $Q_2$ indefinite.
If $Q_1(\bm{\xi})=0$ for every $\bm{\xi}$ such that $Q_2(\bm{\xi})=0,$
then there exists $\lambda\in\mathbb R$ such that $Q_1=\lambda Q_2.$
\end{theorem}
From the fact that
$$\cof(\BS)_{33}=s_{11}s_{22}-s_{12}^2\geq 0,$$
we have $s_{12}(\By)=0$ whenever $s_{11}(\By)=0.$ As $s_{11}$ is indefinite, we have by Marcellini's theorem that
\begin{equation}
\label{4.17}
s_{12}=a_{12}s_{11}\quad\text{for some}\quad a_{12}\in\mathbb R.
\end{equation}
We aim to prove next that $s_{22}$ is a multiple of $s_{11}$ as well. To that end we recall another lemma proven in [\ref{bib:Har.Mil.3}].
\begin{lemma}
\label{le:4.4}
Assume $Q(\bm{\xi})$ is an indefinite quadratic form in $n$ variables that vanishes at a point $\bm{\xi^0}=(\xi_1^0,\xi_2^0,\dots,\xi_n^0).$ Then given any open neighbourhood $U$ of the point $\bm{\xi^0}$ there exist two open subsets $U_1,U_2\subset U$ such that
$$Q(\bm{\xi})<0,\quad \bm{\xi}\in U_1\quad\text{and}\quad Q(\bm{\xi})>0,\quad \bm{\xi}\in U_2.$$
\end{lemma}
Let us prove that $s_{11}(\By)=0$ implies $s_{22}(\By)=0.$ Assume in contradiction $s_{11}(\By^0)=0$ and $s_{22}(\By^0)\neq 0$ for some $\By^0\in\mathbb R^3.$ Let the open neighborhood $U$ of $\By^0$ be such that $s_{22}$ does not vanish and does not change sign in $U.$ Then by Lemma~\ref{le:4.4}, the form $s_{11}$ admits both positive and negative values within $U,$ thus we can find a point $\By^1\in U$ such that $s_{11}(\By^1)s_{22}(\By^1)<0,$ which contradicts the condition $\cof(\BS)_{33}=s_{11}s_{22}-s_{12}^2\geq 0$. Consequently $s_{11}(\By)=0$ implies $s_{22}(\By)=0,$ thus again by Marcellini's theorem above we get
\begin{equation}
\label{4.18}
s_{22}=a_{22}s_{11}\quad\text{ for some}\quad a_{22}\in\mathbb R.
\end{equation}
The above analysis carried out for all other diagonal elements of the cofactor matrix $\cof(\BS)$ yields the form of the matrix $\BS:$
\begin{equation}
\label{4.19}
 \BS(\By)=s_{11}(\By)\cdot
\begin{bmatrix}[1.5]
a_{11} & a_{12} & a_{13} \\
a_{12} & a_{22} & a_{23} \\
a_{13} & a_{23} & a_{33}
\end{bmatrix},
\end{equation}
where the matrix $\BA=(a_{ij})$ has zero determinant and nonnegative second order principal minors, thus by Silvester's criterion it is either positive or negative semidefnite. In both cases we have by (\ref{4.16}) that
\begin{equation}
\label{4.20}
s_{11}(\By)\cdot [\BA\colon\cof(\BT^1(\By))]=\frac{\lambda_2+\lambda_3-2\lambda_1}{\lambda_1\lambda_2\lambda_3}\det(\BT(\By)),
\end{equation}
where the right hand side of (\ref{4.20}) takes strictly positive values a.e. in $\mathbb R^3,$ while the left hand side does not, due to the constant sign of the inner product $\BA\colon\cof(\BT^1(\By)),$ the semi-definiteness of $s_{11},$ and Lemma~\ref{le:4.4}.
This does finish Case 1b.\\
\textbf{Case 2.} Like in Case 1 we first prove that all diagonal entries of $\cof(\BS)$ are nonnegative. For the entry $\cof(\BS)_{33}$ we have that it is either identically zero, or if not then the zero set $\{\By\in\mathbb R^3 \ : \ \cof(\BS)_{33}(\By)=0\}$ has zero Lebesgue measure, thus the steps leading to (\ref{4.15}) go through and we obtain by (\ref{4.15}) that
\begin{equation}
\label{4.21}
\cof(\BS(\By))_{33}\cdot[(-r,-q,1)\BT^1(\By)(-r,-q,1)^T]\equiv 0.
\end{equation}
For the points $\By\in\mathbb R^3$ with $\cof(S(\By))_{33}\neq 0$ we will get that $(-r,-q,1)\BT^1(\By)(-r,-q,1)^T=0,$ hence, as the vector $(-r,-q,1)\neq 0$ we obtain $\det(\BT^1(\By))=0$ by the positive semi-definiteness of $\BT^1.$ Consequently as the zero determinant set
$\{\By\ : \ \det(\BT^1(\By))=0\}$ is the same as the set $\{\By\ : \ \det(\BT(\By))=0\},$ which is a null set, then $\cof(\BS(\By))_{33}\equiv0$.
This argument yields the conditions:
\begin{equation}
\label{4.22}
\cof(\BS(\By))_{11}=\cof(\BS(\By))_{22}=\cof(\BS(\By))_{33}\equiv 0.
\end{equation}
Once (\ref{4.22}) is established we will get the desired results following the steps in Case 1a and Case 1b, where (\ref{4.16}) and the fact that the coefficient $\frac{\lambda_2+\lambda_3-2\lambda_1}{\lambda_1\lambda_2\lambda_3}$ on the right hand side is positive were used.\\
\textbf{Case 3.} The proof of this case immediately follows from (\ref{4.5}), (\ref{4.6}), and (\ref{4.8}). Indeed bearing in mind that $\alpha>0$ and $\det(\BT^1)=\alpha\det(\BT)>0$ a.e. in $\mathbb R^3$, putting together
(\ref{4.5}), (\ref{4.6}), and (\ref{4.8}) and making the change of variables $t=\frac{\lambda}{\lambda_1}$ we obtain
$$\det((\lambda_1\BT^1)^{-1/2}\BT(\lambda_1\BT^1)^{-1/2}-t\BI)=(1-t)^3,\qquad t\in\mathbb R, \ \By\in\mathbb R^3.$$
The last identity implies that the diagonal form of the symmetric matrix $(\lambda_1\BT^1)^{-1/2}\BT(\lambda_1\BT^1)^{-1/2}$ must coincide with the identity matrix for all $\By\in\mathbb R^3,$ thus we get $\BT=\lambda_1\BT^1,$ i.e., $\BT^1$ is a multiple of $\BT.$\\

\end{proof}

\section{Proof of Theorem 2.2}
\label{sec:4}
\setcounter{equation}{0}

\begin{proof}[Proof of Theorem 2.2] Let us mention that some parts of the proof are borrowed from [\ref{bib:Har.Mil.3}] with
minor changes, but we choose to repeat them here for the convenience of the reader. Assume the quadratic form $f(\bm{\xi})=\bm{\xi}\BC \bm{\xi}^T=\Bx^T \BT(\By)\Bx$ is quasiconvex such that $\det(\BT(\By))\equiv 0$ for $\By\in\mathbb R^3.$ We will basically prove here that if the entries of a symmetric matrix 
$\BT(\By)\in\mathbb R^{3\times 3}$ are quadratic forms in $\By\in\mathbb R^3,$ such that $\BT(\By)$ is positive semidefinite for all $\By\in \mathbb R^3$ and $\det(\BT(\By))\equiv 0,$ then the biquadratic form $f(\Bx\otimes\By)=\Bx^T \BT(\By)\Bx$ is a sum of squares. We can without loss of generality assume that the third row of $\BT$ is a linear combination of the first two for a.e. $\By\in\mathbb R^3$ with rational coefficients
$a(\By)$ and $b(\By),$ thus due to the symmetry, the matrix $\BT(\By)$ must have the form
\begin{equation}
\label{3.1}\BT(\By)=
\begin{bmatrix}[1.5]
t_{11} & t_{12} & at_{11}+bt_{12}\\
t_{12} & t_{22} & at_{12}+bt_{22}\\
at_{11}+bt_{12} & at_{12}+bt_{22} & a^2t_{11}+b^2t_{22}+2abt_{12}
\end{bmatrix},
\end{equation}
where the rational functions $a$ and $b$ are given by
\begin{equation}
\label{3.2}
a=\frac{\cof(\BT)_{13}}{\cof(\BT)_{33}},\qquad b=-\frac{\cof(\BT)_{23}}{\cof(\BT)_{33}}.
\end{equation}
For the form $f$ we get
\begin{align}
\label{3.3}
f(\Bx\otimes\By)&=x_1^2t_{11}+2x_1x_2t_{12}+x_2^2t_{22}+2x_1x_3(at_{11}+bt_{12})\\ \nonumber
&+2x_2x_3(at_{12}+bt_{22})+x_3^2(a^2t_{11}+b^2t_{22}+2abt_{12}).
\end{align}
Next we have from the fact $\det{\BT(\By)}=0$ for all $y\in\mathbb R^3$ that
\begin{equation}
\label{3.4}
\mathrm{rank}(\cof(\BT(\By))\leq 1,\quad\text{for all}\quad \By\in\mathbb R^3.
\end{equation}
In order to make sense of (\ref{3.2}) with no fear about the denominators vanishing, we need to consider the case $\cof(\BT(\By))_{33}\equiv 0$ for all $\By\in\mathbb R^3$ separately. Thus assume first it is the case. Observe that if one of the diagonal elements of $\BT(\By)$ is identically zero, then by the positive semidefiniteness of $\BT(\By)$ the elements in the same row and column of $\BT(\By)$ must be identically zero too, thus
$f(\Bx\otimes \By)$ becomes a $2\times 2$ form, and its quasiconvexity automatically implies convexity. Assume now that all diagonal entries of $\BT(\By)$ are nonzero positive semidefinite quadratic forms in $\By\in\mathbb R^3.$ By the positive semi-definiteness of $\cof(\BT(\By)),$ we then have
\begin{equation}
\label{3.5}
\cof(\BT(\By))_{33}=\cof(\BT(\By))_{32}=\cof(\BT(\By))_{31}\equiv 0,\quad \By\in\mathbb R^3.
\end{equation}
The conditions (\ref{3.5}) imply that the matrix obtained from $\BT(\By)$ by removing the last row has rank at most one. We will obtain a more
explicit form of $\BT(\By)$ by means of the following representation lemma proven in [\ref{bib:Har.Mil.3}].
\begin{lemma}
\label{lem:3.1}
Assume $\Bx\in\mathbb R^d$ and $\BA(\Bx)=(a_{ij}(\Bx)),$ $i=1,\dots, m,\  j=1,\dots, n$ is an $m\times n$ matrix with polynomial coefficients, such that each entry $a_{ij}(x)$ is a homogeneous polynomial of degree $2p,$ where $m,n,d,p\in\mathbb N.$ If $\mathrm{rank}(\BA(\Bx))\leq 1$ for all $\Bx\in\mathbb R^d$, then there exist homogeneous polynomials $b_i(\Bx)$ and $c_i(\Bx),$
such that $a_{ij}(\Bx)=b_i(\Bx)c_j(\Bx),$ for $i=1,\dots, m, \ j=1,\dots, n.$
\end{lemma}
If $t_{11}(\By)$ is irreducible in the field of reals, then by Lemma~\ref{lem:3.1} and the obtained rank conditions, the matrix $\BT(\By)$ must have the form
\begin{equation}
\label{3.6}
\begin{bmatrix}[1.5]
P & \alpha P & \beta P\\
\alpha P & \alpha^2 P & \alpha\beta P\\
\beta P & \alpha\beta P & Q
\end{bmatrix},\quad \alpha,\beta\in\mathbb R,
\end{equation}
where $P$ and $Q$ are nonzero positive semidefinite quadratic forms and $\alpha,\beta\in\mathbb R$ with $\alpha\neq 0.$
In the same way if $t_{11}(\By)$ is reducible in the field of reals, then it must be the square of a linear form, thus again by Lemma~\ref{lem:3.1} and the obtained rank conditions the matrix $\BT(\By)$ must have the form
\begin{equation}
\label{3.7}
 \begin{bmatrix}[1.5]
l_1^2 & l_1l_2 & l_1l_3\\
l_1l_2 & l_2^2 & l_2l_3\\
l_1l_3 & l_2l_3 & Q
\end{bmatrix},
\end{equation}
where $Q$ is a nonzero positive semidefinite quadratic form and $l_1,l_2,l_3$ are linear forms with $l_1$ and $l_2$ nonzero.
In the situation of (\ref{3.6}) we have that
$$0\leq \cof(\BT(\By))_{11}=\alpha^2P(\By)(Q(\By)-\beta^2P(\By)),$$
thus
$$Q(\By)-\beta^2P(\By)\geq 0,$$
which implies, that the form $Q(\By)-\beta^2P(\By)$ is convex, i.e.,
\begin{equation}
\label{3.8}
Q(\By)=\beta^2P(\By)+R(\By),
\end{equation}
where $R(\By)$ is convex. Consequently we get
$$f(\Bx\otimes \By)=P(\By)(x_1+\alpha x_2+\beta x_3)^2+R(\By)x_3^2$$
is polyconvex. In the situation of (\ref{3.7}), similarly we get that
$$Q(\By)=l_3^2(\By)+R(\By),$$
where $R$ is convex and thus
$$f(\Bx\otimes\By)=(l_1(\By)x_1+l_2(\By)x_2+l_3(\By)x_3)^2+R(\By)x_3^2,$$
and is thus polyconvex. In what follows we will assume that all diagonal entries of both matrices $\BT(\By)$ and $\cof(\BT(\By))$ are nonzero. Note that this in particular implies that any of the three rows of $\BT$ is a linear combination of the remaining two, as we have for the third row in (\ref{3.1}). The rank condition (\ref{3.4}) implies by Lemma~4.1 that the cofactor matrix $\cof(\BT)$ has the form
\begin{equation}
\label{3.9}\cof(\BT)=\\
\begin{bmatrix}[1.5]
c_1d_1 & c_1d_2 & c_1d_3\\
c_2d_1 & c_2d_2 & c_2d_3\\
c_3d_1 & c_3d_2 & c_3d_3
\end{bmatrix},
\end{equation}
for some homogeneous polynomials $c_i(\By)$ and $d_i(\By),$ $i=1,2,3$ and $\By\in\mathbb R^3.$ As all diagonal entries of $\cof(\BT(\By))$ are polynomials of degree four, and $\deg(c_id_j)\leq 4$ for $i,j=1,2,3,$ we must in fact have
\begin{equation}
\label{3.10}
\deg(c_id_j)=4,\quad\text{for all}\quad i,j=1,2,3.
\end{equation}
The cofactor matrix $\cof(\BT(\By))$ must be positive semidefinite for all $\By\in\mathbb R^3$ given that $\BT(\By)$ is such, thus we get the set of inequalities
\begin{equation}
\label{3.11}
c_i(\By)d_i(\By)\geq 0,\quad\text{for all}\quad \By\in\mathbb R^3,\ \ i=1,2,3.
\end{equation}
Next we aim to come up with a more explicit form of $\cof(\BT(\By))$ using the obtained conditions (\ref{3.9})-(\ref{3.11}). To that end we consider the cases $\deg(c_1)=0,1,2$ separately (note that the case $\deg(c_1)>2$ implies $\deg(d_1)<2$ and we
can consider $d_1$ instead of $c_1).$\\
\textbf{Case 1: $\deg(c_1)=0$ }.\\
\textbf{Case 2: $\deg(c_1)=1$ }.\\
\textbf{Case 3: $\deg(c_1)=2$ }.\\
Next we examine each case in detail.\\
\textbf{Case 1.} In this case we have from (\ref{3.10}) that
$$
c_i(\By)=c_i\in\mathbb R,\quad\deg(d_i)=4\quad \text{for all}\quad i=1,2,3,
$$
which gives by (\ref{3.2}) $a(\By)=\frac{c_1}{c_3}=a\in\mathbb R$ and $b(\By)=-\frac{c_2}{c_3}=b\in\mathbb R.$
Consequently we get by (\ref{3.3})
\begin{equation}
\label{3.12}
f(\Bx\otimes\By)=t_{11}(\By)(x_1+ax_3)^2+2t_{12}(\By)(x_1+ax_3)(x_2+bx_3)+t_{22}(\By)(x_2+bx_3)^2.
\end{equation}
Terpstra has proven in [\ref{bib:Terpstra}] the following classical result.
\begin{theorem}[Terpstra]
\label{th:3.1}
Any $2\times n$ (or $n\times 2$) quasiconvex quadratic form is polyconvex.
\end{theorem}
Note that Terpstra's theorem implies that any $2\times n$ (or $n\times 2$) nonnegative biquadratic form $Q(\Bx,\By)$ is in fact a sum of squares.
Introducing the new independent variables $X_1=x_1+ax_3$ and $X_2=x_2+bx_3$ we have that the $2\times 3$ biquadratic form
$$g(\BX,\By)=t_{11}(\By)X_1^2+2t_{12}(\By)X_1X_2+t_{22}(\By)X_2^2$$
is nonnegative, thus by Terpstra's theorem above $g$ must be the sum of squares of 2-homogeneous forms that are linear combinations of $X_iy_j,$ i.e., $f$ is polyconvex.\\
\textbf{Case 2.} On one hand we similarly have from (\ref{3.10}) that
\begin{equation}
\label{3.13}
\deg(c_i)=1,\quad\deg(d_i)=3\quad \text{for all}\quad i=1,2,3.
\end{equation}
On the other hand we have from (\ref{3.11}) that $c_i$ must divide $d_i$ for each $1\leq i\leq 3,$ thus we get the form of $\cof(\BT):$
\begin{equation}
\label{3.14}\cof(\BT)=
\begin{bmatrix}[1.5]
c_1^2P & c_1c_2Q & c_1c_3R\\
c_1c_2P & c_2^2Q & c_2c_3R\\
c_1c_3P & c_2c_3Q& c_3^2R
\end{bmatrix},
\end{equation}
where $c_i(\By)$ are linear forms and $P(\By),Q(\By),R(\By)$ are positive semidefinite quadratic forms.
Again we get from (\ref{3.2}) that $a(\By)=\frac{c_1(\By)}{c_3(\By)}$ and $b(\By)=-\frac{c_2(\By)}{c_3(\By)},$ and utilizing (\ref{3.3}):
\begin{equation}
\label{3.15}
f(\Bx\otimes\By)=\frac{1}{c_3^2}\left[t_{11}(\By)(x_1c_3+x_3c_1)^2+2t_{12}(\By)(x_1c_3+x_3c_1)(x_2c_3-x_3c_2)+t_{22}(\By)(x_2c_3-x_3c_2)^2\right].
\end{equation}
The goal is to show that we can obtain necessary factorizations and abbreviations to end up with a $2\times 3$ nonnegative biquadratic form that then must be polyconvex by Terpstra's theorem. Consider the biquadratic form $g$ in the variables $\By$ and $\BX=(X_1,X_2)$ given by
\begin{equation}
\label{3.16}
g(\BX,\By)=t_{11}(\By)X_1^2+2t_{12}(\By)X_1X_2+t_{22}(\By)X_2^2.
\end{equation}
We have that $g(x_1c_3+x_3c_1,x_2c_3-x_3c_2,\By)=\frac{f(\Bx\otimes\By)}{c_3^2}\geq 0$
for all $\Bx,\By\in\mathbb R^3$ such that $c_{3}(\By)\neq 0,$ thus we obtain
\begin{equation}
\label{3.17}
g(x_1c_3+x_3c_1,x_2c_3-x_3c_2,\By)\geq 0,\quad\text{for all}\quad \Bx,\By\in\mathbb R^3,
\end{equation}
by continuity as the set $c_3(\By)=0$ is just a hyperplane in $\mathbb R^3.$
For any fixed $\By\in\mathbb R^3$ such that $c_3(\By)\neq0$ and for any fixed values $X_1,X_2,$ we can find values
$x_1,x_2,x_3\in\mathbb R$ such that $x_1c_3+x_3c_1=X_1$ and $x_2c_3-x_3c_2=X_2$ (for instance take $x_1=\frac{X_1}{c_3(\By)},$
$x_2=\frac{X_2}{c_3(\By)},$ and $x_3=0$), thus again by continuity we get from (\ref{3.17}) the condition
\begin{equation}
\label{3.18}
g(\BX,\By)=t_{11}(\By)X_1^2+2t_{12}(\By)X_1X_2+t_{22}(\By)X_2^2\geq 0\quad\text{for all}\quad \BX\in\mathbb R^2, \By\in\mathbb R^3.
\end{equation}
Hence $g$ is a $2\times 3$ nonnegative biquadratic form, and by Terpstra's theorem it must be a sum of squares of 2-homogeneous forms in $X_iy_j:$
\begin{equation}
\label{3.19}
g(\BX,\By)=t_{11}(\By)X_1^2+2t_{12}(\By)X_1X_2+t_{22}(\By)X_2^2=\sum_{k=1}^6(a_k(\By)X_1+b_k(\By)X_2)^2,
\end{equation}
where $a_k(\By)$ and $b_k(\By)$ are linear forms in $\By.$ Thus we discover
\begin{equation}
\label{3.20}
f(\Bx\otimes\By)=\frac{1}{c_3^2}\sum_{k=1}^6(a_k(x_1c_3+x_3c_1)+b_k(x_2c_3-x_3c_2))^2=\frac{1}{c_3^2}\sum_{k=1}^6[c_3(a_kx_1+b_kx_2)+x_3(a_kc_1-b_kc_2)]^2.
\end{equation}
Equating the coefficients of $x_3^2$ in the original form of $f$ and in (\ref{3.20}) we get the key equality
\begin{equation}
\label{3.21}
\sum_{k=1}^6(a_kc_1-b_kc_2)^2=c_3^2t_{33}.
\end{equation}
The condition (\ref{3.21}) in particular implies that $a_k(\By)c_1(\By)-b_k(\By)c_2(\By)=0$ for all $k=1,2,\dots,6$ whenever $c_3(\By)=0.$ Since $c_3$ is linear, this means that $c_3$ divides $ a_kc_1-b_kc_2$ for all $k,$ thus we get
\begin{equation}
\label{3.22}
a_k(\By)c_1(\By)-b_k(\By)c_2(\By)=h_k(\By)c_3(\By),\quad k=1,2,\dots,6,
\end{equation}
for some linear forms $h_k(\By).$ Plugging the obtained forms of $a_kc_1+b_kc_2$ back into (\ref{3.20}), we obtain
\begin{equation}
\label{3.23}
f(\Bx\otimes\By)=\sum_{k=1}^6(a_kx_1+b_kx_2+h_kx_3)^2,
\end{equation}
 i.e, $f$ is polyconvex.\\
\textbf{Case 3.} We have due to (\ref{3.6}) that $\deg(d_i)=\deg(c_i)=2$ for all $i=1,2,3.$ The following two cases are qualitatively different:\\
\textbf{Case 3a: $c_3(\By)$ is indefinite.}\\
\textbf{Case 3b: $c_3(\By)$ is positive semidefinite.}\\
\textbf{Case 3a.} It is easy to verify that the steps (\ref{3.15})-(\ref{3.21}) go through in this case too, thus we have
\begin{equation}
\label{3.24}
\sum_{k=1}^6(a_kc_1-b_kc_2)^2=c_3^2t_{33},
\end{equation}
where $a_k$ and $b_k$ are linear forms for $k=1,2,\dots,6.$ Equality (\ref{3.24}) and positivity of both sides imply that
\begin{equation}
\label{3.25}
a_k(\By)c_1(\By)-b_k(\By)c_2(\By)=0,\quad \text{whenever}\quad c_3(\By)=0,\quad\text{for some}\quad \By\in\mathbb R^3.
\end{equation}
Next we prove the following simple lemma.
\begin{lemma}
\label{le:3.3}
Let $Q(\By)$ be an indefinite quadratic form in the variable $\By\in\mathbb R^3$ and let $P(\By)$ be a third order homogeneous polynomial. If
$P(\By)=0$ whenever $Q(\By)=0$ for some $\By\in\mathbb R^3,$ then $Q$ must divide $P.$
\end{lemma}

\begin{proof}[Proof of Lemma~4.3]
As $Q$ is indefinite, we can without loss of generality assume that it has one of the normal forms:
$$Q(\By)=y_1^2-y_2^2,\quad\text{or}\quad Q(\By)=y_1^2+y_2^2-y_3^2. $$
In the first case we have $Q=(y_1-y_2)(y_1+y_2)$ and thus $P(\By)=0$ whenever one of the linear forms $y_1-y_2$ or $y_1+y_2$ vanishes, thus obviously $P$ is divisible by both, and hence by their product too (by the unique factorization over the field of reals). In the case
$Q(\By)=y_1^2+y_2^2-y_3^2$ we can separate the multiple of $y_3^2$ within $P$ and write
\begin{equation}
\label{3.26}
P(\By)=(y_1^2+y_2^2-y_3^2)l(\By)+y_3\varphi(y_1,y_2)+\psi(y_1,y_2),
\end{equation}
where $l$ is a linear form in $\By$, and $\varphi$ and $\psi$ are homogeneous forms in $y_1$ and $y_2$ of degree two and three, respectively.
Now for any $y_1,y_2\in\mathbb R$ such that $y_1^2+y_2^2>0,$ by choosing $y_3=\pm\sqrt{y_1^2+y_2^2}$ we get the system
\begin{equation}
\label{3.27}
\pm\sqrt{y_1^2+y_2^2}\cdot\varphi(y_1,y_2)+\psi(y_1,y_2)=0,
\end{equation}
which implies first $\psi(y_1,y_2)=0$ and then $\varphi(y_1,y_2)=0,$ i.e., $\psi$ and $\varphi$ identically vanish and thus $P=Ql.$

\end{proof}

Consequently, applying the lemma for the pairs $a_kc_1-b_kc_2$ and $c_3$ we get
\begin{equation}
\label{3.28}
a_kc_1-b_kc_2=c_3h_k,\quad k=1,2,\dots,6,
\end{equation}
for some linear forms $h_k$. Owing back to the form of $f$ in (\ref{3.20}) we arrive at
$$
f(\Bx\otimes\By)=\sum_{k=1}^6(a_kx_1+b_kx_2+h_kx_3)^2,
$$
utilizing (\ref{3.28}), i.e., $f$ is polyconvex.\\
\textbf{Case 3b.} We can assume without loss of generality that all forms $c_i$ and $d_i$ are semidefinite (positive or negative). We divide this case further into two possible cases.\\
\textbf{Case 3ba: $c_3(\By)$ is reducible in the field of reals.}\\
\textbf{Case 3bb: $c_3(\By)$ is irreducible in the field of reals.}\\
\textbf{Case 3ba}. As $c_3$ is semidefinite and reducible, it must be a multiple of the square of a linear form, i.e, $c_3(\By)=\sigma l^2(\By),$
where $\sigma\in\{-1,1\}$ and $l(\By)$ is linear. Again, we can easily verify that the steps (\ref{3.15})-(\ref{3.21}) go through in this case too, thus we have
\begin{equation}
\label{3.29}
\sum_{k=1}^6(a_kc_1-b_kc_2)^2=c_3^2t_{33}=l^4t_{33},
\end{equation}
where $a_k$ and $b_k$ are linear forms for $k=1,2,\dots,6.$ Equality (\ref{3.29}), the linearity of $l,$ and positivity of both parts of the equality imply that all the $3-$homogeneous forms $a_kc_1-b_kc_2$ contain a factor of $l$ for all $k.$ After factoring an $l^2$ out in (\ref{3.29}) we get for the same reason that all of the forms $a_kc_1-b_kc_2$ contain a factor of $l^2,$ i.e., $c_3.$ Denoting $a_kc_1-b_kc_2=c_3h_k$ for $k=1,2,\dots,6$ we again end up with the form
$$
f(\Bx\otimes\By)=\sum_{k=1}^6(a_kx_1+b_kx_2+h_kx_3)^2,
$$
by (\ref{3.20}), hence $f$ is polyconvex.\\
\textbf{Case 3ba.} We assume without loss of generality that $c_i$ and $d_i$ are all irreducible and semidefinite for $i=1,2,3.$ From the fact that the cofactor matrix $\cof(\BT)$ is symmetric, we have the set of equalities
\begin{equation}
\label{3.30}
c_id_j=c_jd_i\quad\text{for all}\quad i,j=1,2,3.
\end{equation}
In the case when $c_1$ and $d_1$ are linearly independent, we get from the equality $c_1d_i=d_1c_i,$ $i=2,3,$ from the irreducibility of the factors in it, and from the unique factorization of homogeneous polynomials in the field of reals, that $c_2=\alpha c_1$ and $c_3=\beta c_1$ for some nonzero $\alpha,\beta\in\mathbb R.$ Again, the steps (\ref{3.15})-(\ref{3.21}) go through and thus (\ref{3.20}) yields the polyconvex form for $f:$
$$
f(\Bx\otimes\By)=\sum_{k=1}^6[a_kx_1+b_kx_2+\frac{x_3}{\beta}(a_k-\alpha b_k)]^2.
$$
In what follows we can assume that $c_i$ and $d_i$ are linearly dependent for all $i=1,2,3.$ After a change of sign in all of $c_i$ and $d_i$ if necessary, multiplying all of $d_i$ by the same factor we can assume that all of $c_i$ are irreducible, positive semidefnite, and
\begin{equation}
\label{3.30}
d_i=c_i\quad\text{for all}\quad i=1,2,3.
\end{equation}
Recall next that by (\ref{3.9}) we have 
\begin{equation}
\label{3.31}
c_3^2+t_{12}^2=t_{11}t_{22},\qquad c_2^2+t_{13}^2=t_{11}t_{33},\qquad c_1^2+t_{23}^2=t_{22}t_{33},
\end{equation}
where we aim at examining the above conditions by factoring in the complex field as
\begin{equation}
\label{3.32}
(t_{12}+ic_3)(t_{12}-ic_3)=t_{11}t_{22}.
\end{equation}
Because each of the factors above is a second order homogeneous polynomial in $\By,$ the factors on the right hand side are purely real and the factors on the left hand side have nonzero imaginary parts, then again by the unique factorization in the field of complex numbers,
we discover that both multipliers $t_{12}+ic_3$ and $t_{12}-ic_3$ must be reducible, and thus both $t_{11}$ and $t_{22}$ must split into the product of two linear forms with complex coefficients as well, i.e.,
\begin{equation}
\label{3.33}
t_{11}=(u_1+iu_1')(v_1-iv_1'),\qquad
t_{22}=(u_2+iu_2')(v_2-iv_2'),
\end{equation}
where $u_1,v_1,u_1',v_1'$ are real linear forms. From the fact that $t_{11}$ is real we get $u_1v_1'=v_1u_1',$ which eventually implies from the unique factorization and positive semidefiniteness of $t_{11}$ that we can choose $u_1,v_1,u_1',v_1'$ so that $u_1=v_1$ and $u_1'=v_1',$
which also gives $t_{11}=(u_1+iu_1')(u_1-iu_1').$ This observation for all $k=1,2,3$ will lead to the forms
\begin{equation}
\label{3.34}
t_{kk}=u_k^2+u_k'^2,\quad k=1,2,3.
\end{equation}
Also, it is straightforward that examining all the equalities in (\ref{3.31}) like in the steps (\ref{3.32})-(\ref{3.33}) we get the possible forms for $t_{ij}:$
\begin{equation}
\label{3.35}
t_{ij}=u_iu_j+\sigma_{ij}u_i'u_j',\quad 1\leq i<j\leq3,
\end{equation}
where $\sigma_{ij}\in \{-1,1\},$ and $\sigma_{ij}=\sigma_{ji}$ for $1\leq i<j\leq3$. In the case when an even number of $\sigma_{ij},$ $(i\neq j)$ are negative (which means zero or two of them), then by changing the sign of one of the linear forms $u_i',$ we can assume without loss of generality that all $\sigma_{ij}$ are positive. 
Hence taking into account the forms of $t_{ij}$ in (\ref{3.34}) and (\ref{3.35}), we get
\begin{align*}
f(\Bx\otimes\By)&=\sum_{i,j=1}^3x_ix_jt_{ij}(\By)\\
&=(x_1u_1(\By)+x_2u_2(\By)+x_3u_3(\By))^2+(x_1u_1'(\By)+x_2u_2'(\By)+x_3u_3'(\By))^2,
\end{align*}
i.e., $f$ is polyconvex. It remains to analysis the case when an odd number of $\sigma_{ij},$ $(i\neq j)$ are negative, in which case we can again assume without loss of generality that $\sigma_{ij}=-1$ for all $1\leq i<j\leq3.$ In this situation we need to further analyze the off-diagonal elements of the cofactor matrix in (\ref{3.9}). Recall that we have the following identities:
\begin{equation}
\label{3.37}
\begin{cases}
t_{kk}=u_k^2+u_k'^2, & \quad k=1,2,3,\\
t_{ij}=u_iu_j-u_i'u_j', & \quad 1\leq i<j\leq3,\\
c_{k}=u_iu_j'+u_ju_i', & \quad \{i,j,k\}=\{1,2,3\}.
\end{cases}
\end{equation}
Consequently, given the values of the entries of $\BT,$ and the values of $c_i$ in (\ref{3.37}), for the off-diagonal elements of the cofactor matrix $\cof(\BT)$ we get the following set of identities:
\begin{equation}
\label{3.38}
\begin{cases}
u_1'(u_2'u_3'-u_2u_3)=u_1(u_2u_3'+u_2'u_3),\\
u_2'(u_1'u_3'-u_1u_3)=u_2(u_1u_3'+u_1'u_3),\\
u_3'(u_1'u_2'-u_1u_2)=u_3(u_1u_2'+u_1'u_2).\\
\end{cases}
\end{equation}
Note first that if the linear forms $u_1$ and $u_1'$ are collinear, i.e., $u_1'=\lambda u_1$ for some $\lambda\in\mathbb R,$ then we can separate the variable $x_1u_1$
from the bi-quadratic form $f(\Bx\otimes\By)$ by extracting a perfect square, ending up with anther nonnegative form g, depending only on the vector variables 
$(x_2,x_3)$ and $\By.$ Hence by Terpstra's theorem (Theorem~4.2), the form $g$, being $2\times 3,$ must be a sum of squares, thus so must be $f$ too, i.e., $f$ is polyconvex. In what follows we assume that none of the pairs $(u_i,u_i')$ is linearly dependent, in particular none of the forms $u_i$ and $u_i'$ is identically zero
for $i=1,2,3.$ Observe that as $u_1$ and $u_1'$ are linear forms, then from the first equality in (\ref{3.38}) we get that $u_2'u_3'-u_2u_3$ must be divisible by $u_1$ and $u_2u_3'+u_2'u_3$ must be divisible by $u_1'.$ 
Let $u_2'u_3'-u_2u_3=v_1u_1,$ thus we also have $u_2u_3'+u_2'u_3=v_1u_1',$ for some linear form $v_1=v_1(\By).$ By analogous observations for the second and third identities in (\ref{3.38}) we arrive at the set of equalities: 
\begin{equation}
\label{3.39}
\begin{cases}
u_2'u_3'-u_2u_3=v_1u_1, \quad u_2u_3'+u_2'u_3=v_1u_1',\\
u_1'u_3'-u_1u_3=v_2u_2,\quad u_1u_3'+u_1'u_3=v_2u_2',\\
u_1'u_2'-u_1u_2=v_3u_3,\quad u_1u_2'+u_1'u_2=v_3u_3',\\
\end{cases}
\end{equation}
for some linear forms $v_1,v_2$, and $v_3.$ Multiplying the first identity by $u_2',$ the second identity by $u_2$ in the first row of (\ref{3.39}) and adding the obtained equalities together we get 
$$u_3'(u_2^2+u_2'^2)=v_1(u_1u_2'+u_1'u_2).$$
Consequently, utilizing the second identity in the third row of (\ref{3.39}), and keeping in mind that the form $u_3'$ is nonzero, we derive from the last equality: 
$$u_2^2+u_2'^2=v_1v_3.$$
Finally, note that the lest equality is equivalent to 
$$(u_2+iu_2')(u_2-iu_2')=v_1v_3,$$
which by the unique factorization is possible if and only if the forms $u_2$ and $u_2'$ are collinear, i.e., we are back in the first situation and hence $f$ is polyconvex. 
This completes the proof of Theorem~\ref{th:2.2}

\end{proof}


\begin{thebibliography}{999}

\bibitem{2} G. Allaire and R.V. Kohn. Optimal lower bounds on the elastic energy of a composite made from two non-well-ordered isotropic materials,
             \textit{Quarterly of applied mathematics, } vol. LII, 311-333 (1994)
\label{bib:All.Koh.}

\bibitem{} E. Artin. \"Uber die Zerlegung definiter Funktionen in Quadrate, \textit{Abhandlungen aus dem Mathematischen Seminar der Universit\"at Hamburg.} 5 (1): 100--115, 1927.
\label{bib:Artin}


\bibitem{3}  J. M. Ball. Convexity conditions and existence theorems in nonlinear elasticity, \textit{Arch. Ration. Mech. Anal,} 63, 337-403 (1976).
\label{bib:Ball}

\bibitem{3} J. M. Ball and R. D. James. Fine phase mixtures as minimizers of energy, \textit{Arch. Rational Mech. Anal., } 100(1):13--52, 1987.
\label{bib:Bal.Jam.}

\bibitem{3}  B. Bene\v{s}ov\'a and M. Kru\v{z}\'ik. Weak lower semicontinuity of integral functionals and applications. \textit{SIAM Rev.,} 59(4):703--766, 2017.
\label{bib:Ben.Kru.}


\bibitem{} G. Blekherman. Nonnegative Polynomials and Sums of Squares, \textit{Journal of the AMS,} 25, 2012, 617--635.
\label{bib:Blekherman.1}


\bibitem{} G. Blekherman, G. G. Smith and M. Velasco. Sums of squares and varieties of minimal degree, \textit{Journal of the AMS,} 29 (2016) 893--913.
\label{bib:Ble.Smi.Vel.}


\bibitem{} G. Blekherman, R. Sinn, G. Smith, M. Velasco. Sums of Squares: A Real Projective Story", \textit{Notices of the AMS}, arXiv:2101.05773.
\label{bib:Ble.Sin.Smi.Vel.}


\bibitem{} G. Blekherman. A Brief Introduction to Sums of Squares, \textit{Proceedings of Symposia in Applied Mathematics, AMS}.
\label{bib:Blekherman.2}


\bibitem{3} O. Boussaid, C. Kreisbeck, and A. Schl\"omerkemper. Characterizations of Symmetric Polyconvexity,
\textit{Archive for Rational Mechanichs and Analysis,} vol. 234, pp. 417--451(2019).
\label{bib:Bou.Kre.Sch.}


\bibitem{} A. Buckley and K. \u{S}ivic. Nonnegative biquadratic forms with maximal number of zeros, \textit{preprint,} https://arxiv.org/pdf/1611.09513.pdf
\label{bib:Buc.Siv.}


\bibitem{5} A. Cherkaev. \textit{Variational Methods for Structural Optimization}, Springer Applied Mathematical Sciences, Vol. 140 (2000).
\label{bib:Cherkaev}


\bibitem{}  A. V. Cherkaev and L. V. Gibiansky. The exact coupled bounds for effective tensors of electrical and magnetic properties of two-component two-dimensional composites, \textit{Proceedings of the Royal Society of Edinburgh. Section A, Mathematical and Physical Sciences,} 122 (1992), pp. 93--125.
\label{bib:Che.Gib.1}


\bibitem{} A. V. Cherkaev and L. V. Gibiansky. Coupled estimates for the bulk and shear moduli of a two-dimensional isotropic elastic composite, \textit{Journal of the Mechanics and Physics of Solids,} 41 (1993), pp. 937--980.
\label{bib:Che.Gib.2}

\bibitem{} S. J. Cho. Generalized Choi maps in three-dimensional matrix algebra, \textit{Linear algebra and its applications,} 171: 213--224 (1992).
\label{bib:Cho}

\bibitem{} M.-D. Choi. Positive semidefinite biquadratic forms, \textit{Linear algebra and its applications,} 12, 95--100 (1975).
\label{bib:Choi}

\bibitem{} M.-D. Choi and T.-Y. Lam.  Extremal positive semidefinite forms, \textit{Math. Ann.,} 231, 1--18 (1977).
\label{bib:Cho.Lam.}

\bibitem{6} B. Dacorogna. \textit{Direct methods in the calculus of variations.} Springer Applied Mathematical Sciences, Vol. 78, 2nd Edition (2008).
\label{bib:Dacorogna}


\bibitem{} I. Fonseca and S. M\"uller. A--quasiconvexity, lower semicontinuity, and Young measures. \textit{SIAM J. Math. Anal.,} 30(6):1355--1390, 1999.
\label{bib:Fon.Mue.}


\bibitem{6} D. Harutyunyan. A note on the extreme points of the cone of quasiconvex quadratic forms with orthotropic symmetry, \textit{Journal of Elasticity,} 140, pp. 79--93 (2020).
 \label{bib:Harutyunyan}


 \bibitem{12} D. Harutyunyan and G. W. Milton. Explicit examples of extremal quasiconvex quadratic forms that are not polyconvex,
 \textit{Calculus of Variations and Partial Differential Equations} , October 2015, Vol. 54, Iss. 2, pp. 1575-1589.
 \label{bib:Har.Mil.1}

 \bibitem{12} D. Harutyunyan and G. W. Milton. On the relation between extremal elasticity tensors with orthotropic symmetry and extremal polynomials, \textit{Archive for Rational Mechanics and Analysis}, Vol. 223, Iss. 1, pp 199-212, 2017.
 \label{bib:Har.Mil.2}

 \bibitem{12} D. Harutyunyan and G. W. Milton. Towards characterization of all $3\times 3$ extremal quasiconvex quadratic forms,
 \textit{Communications of Pure and Applied Mathematics}, Vol. 70, Iss. 11, Nov. 2017, pp. 2164-2190.
 \label{bib:Har.Mil.3}
 
\bibitem{} J. W. Helton, S. A. McCullough, and V. Vinnikov. Noncommutative convexity arises from linear matrix inequalities, \textit{J. Funct. Anal.} 240 (2006), 105--191.
\label{bib:Hel.McC.Vin.} 

 \bibitem{} D. Hilbert. \"Uber die Darstellung definiter Formen als Summen von Formenquadraten,
\textit{Math. Ann.,} 32 (1888), 342--350.
 \label{bib:Hilbert}


 \bibitem{12} J. Hou, Ch.-K. Li, Y.-T. Poon, X. Qi, and  N.-S. Sze. A new criterion and a special class of $k-$positive maps, \textit{Linear Algebra and its Applications} 470 (2015) 51-69.
 \label{bib:Hou.Li.Poo.Qi.Sze.}


 \bibitem{7}  H. Kang, E. Kim, and G. W. Milton.
Sharp bounds on the volume fractions of two materials in a two-dimensional body from electrical boundary measurements: the translation method,
\textit{Calculus of Variations and Partial Differential Equations,} 45, 367-401 (2012).
\label{bib:Kan.Kim.Mil.}

\bibitem{8} H. Kang and G. W. Milton. Bounds on the volume fractions of two materials in a
three dimensional body from boundary measurements by the translation method,
\textit{SIAM Journal on Applied Mathematics,} 73, 475--492 (2013).
 \label{bib:Kan.Mil.}

\bibitem{9} H. Kang, G. W. Milton, and J.-N. Wang.  Bounds on the Volume Fraction of the Two-Phase Shallow Shell Using One Measuremen.
\textit{Journal of Elasticity,} 114, 41-53 (2014).
\label{bib:Kan.Mil.Wan.}

\bibitem{} X. Li and W. Wu. A class of generalized positive linear maps on matrix algebras, \textit{Linear algebra and its applications} 439 (2013) 2844-2860.
\label{bib:Li.Wu.}


\bibitem{11} R. V. Kohn and R. Lipton. Optimal bounds for the effective energy of a mixture
of isotropic, incompressible, elastic materials.
\textit{Archive for Rational Mechanics and Analysis,} 102, 331--350 (1988).
\label{bib:Koh.Lip.}


\bibitem{11} P. Marcellini. Quasiconvex quadratic forms in two dimensions. \textit{Appl. Math. Optim.} 11(2), 183--189, 1984.
\label{bib:Marcellini}


 \bibitem{12} G. W. Milton. On characterizing the set of positive effective tensors of composites: The variational method and the translation method,
 \textit{Communications on Pure and Applied Mathematics,} Vol. XLIII, 63-125 (1990).
\label{bib:Milton.1}

 \bibitem{13} G. W. Milton. \textit{The Theory of Composites,} vol. 6 of Cambridge Monographs on Applied and Computational Mathematics, Cambridge University Press, Cambridge, United Kingdom, 2002
\label{bib:Milton.2}

 \bibitem{14} G. W. Milton. Sharp inequalities which generalize the divergence theorem: an extension of the notion of quasi-convexity, \textit{Proceedings Royal Society A} 469, 20130075 (2013).
 \label{bib:Milton.3}

\bibitem{15} G. W. Milton and L. H. Nguyen. Bounds on the volume fraction of 2-phase, 2-dimensional
elastic bodies and on (stress, strain) pairs in composites.
\textit{Comptes Rendus M\'ecanique,} 340, 193-204 (2012).
\label{bib:Mil.Ngu.}


 \bibitem{16} C. B. Morrey. Quasiconvexity and the lower semicontinuity of multiple integrals,
 \label{bib:Morrey.1}               \textit{Pacific Journal of Mathematics} 2, 25-53 (1952)

\bibitem{17} C. B. Morrey. \textit{Multiple integrals in the calculus of variations},
  \label{bib:Morrey.2}             Springer--Verlag, Berlin, 1966.
  
  
\bibitem{17} S. M\"uller. Variational models for microstructure and phase transitions. \textit{In Calculus of variations and geometric
evolution problems (Cetraro, 1996), }volume 1713 of \textit{Lecture Notes in Math.,} pp. 85--210. Springer, Berlin, 1999.
\label{bib:Mueller.1}


 \bibitem{18} F. Murat and L. Tartar. Calcul des variations et homog\'en\'isation. (French) [Calculus of variation and homogenization], in
            Les m\'ethodes de l'homog\'en\'eisation: th\'eorie et applications en physique, volume 57 of Collection de la Direction des
\'etudes et recherches d'Electricit\'e de France, pages 319-369, Paris, 1985, Eyrolles, English
translation in Topics in the Mathematical Modelling of Composite Materials, pages 139-173, ed. by A. Cherkaev and R. Kohn, ISBN 0-8176-3662-5.
\label{bib:Mur.Tar.}

\bibitem{} R. Quarez. On the real zeros of positive semidefinite biquadratic forms, \textit{Comm. Algebra,} 43 (2015), 1317--1353.
\label{bib:Quarez.1}

\bibitem{} R. Quarez. Symmetric determinantal representation of polynomials, \textit{Linear Algebra and its Applications} 436 (2012), 3642--3660.
\label{bib:Quarez.2}

\bibitem{} B. Reznick. On Hilbert's construction of positive polynomials, \textit{preprint,} available at: https://arxiv.org/pdf/0707.2156v1.pdf.
\label{bib:Reznick}


\bibitem{24} C. Scheiderer. Sums of squares of polynomials with rational coefficients. \textit{J. Eur. Math. Soc.} 18(7):1495--1513, 2016.
\label{bib:Scheiderer}

\bibitem{24} D. Serre. Condition de Legendre-Hadamard: Espaces de matrices de rang$\neq  1$.
               (French) [Legendre-Hadamard condition: Space of matrices of rank$\neq  1$], Comptes
Rendus de l'Acad\'emie des sciences, Paris 293, 23-26 (1981).
 \label{bib:Serre}


\bibitem{}A. Stefan and A. Welters. A Short Proof of the Symmetric Determinantal Representation of Polynomials, \textit{preprint,} https://arxiv.org/abs/2101.03589
\label{bib:Ste.Wel.}


\bibitem{19} E. Stormer. Separable states and the structural physical approximation of a positive map, \textit{Journal of Functional Alanysis} 264 (2013) 2197-2205.
 \label{bib:Stormer}
 
 \bibitem{19} V. \v{S}ver\'ak. New examples of quasiconvex functions. \textit{Arch. Ration. Mech. Anal.,} 119(4):293--300, 1992. 
 \label{bib:Sverak.1}

\bibitem{19}  V. \v{S}ver\'ak. Rank-one convexity does not imply quasiconvexity. \textit{Proc. Roy. Soc. Edinburgh Sect. A,} 120(1-2):185--189, 1992.
\label{bib:Sverak.2}


 \bibitem{19} L. Tartar. Compensated compactness and applications to partial differential equations, in Nonlinear Analysis and Mechanics, Heriot-Watt Symposium, Volume IV, edited by R. J. Knops, volume 39 of Research Notes in Mathematics, pages 136-212, London, 1979, Pitman Publishing Ltd.
 \label{bib:Tartar}


 \bibitem{20} F. J. Terpstra. Die Darstellung biquadratischer Formen als Summen von Quadraten
   \label{bib:Terpstra}            mit Anwendung auf die Variationsrechnung, \textit{Mathematische Annalen,} 116 (1938), 166--180.


\bibitem{26} L. Van Hove. Sur l'extension de la condition de Legendre du calcul des variations aux
  \label{bib:VanHove.1}                 int\'egrales multiples \'a plusieurs functions inconnues, \textit{Nederl. Akad. Wetensch. Proc.} \textbf{50} (1947), 18-23.

\bibitem{27} L. Van Hove. Sur le signe de la variation seconde des int\'egrales multiples
  \label{bib:VanHove.2}                \'a plusieurs functions inconnues, \textit{Acad. Roy. Belgique Cl. Sci. M\'em. Coll.} \textbf{24} (1949), 68.


\bibitem{27} K. Zhang. The structure of rank-one convex quadratic forms on linear elastic strains. \textit{Proc. Roy. Soc. Edinburgh Sect. A,} 
133(1):213--224, 2003.
 \label{bib:Zhang.1} 


 \end{thebibliography}
\end{document}